\documentclass[reqno]{amsart}

\usepackage{amsthm, amsmath, amssymb, enumitem, geometry, caption, graphicx,
array, hyperref, url, fontenc, inputenc, babel, booktabs, cite,float, footmisc, multicol, xcolor, newtxmath, newtxtext,kvoptions, nag, ragged2e, pdf14, pdftexcmds, xpatch, zref-base}

\usepackage{mathrsfs}

\usepackage{amsfonts}
\usepackage[normalem]{ulem}

\newcommand\NN{{\mathbb N}}

\newcommand\ZZ{{\mathbb Z}}

\newcommand\RR{{\mathbb R}}

\newcommand\eps{\varepsilon}

\renewcommand\rho{\varrho}

\newcommand\Ups{\Upsilon}

\newcommand{\f}[2]{\frac{#1}{#2}}

\newcommand{\multsum}[2]{\sum_{{\scriptstyle #1}\atop {\scriptstyle #2}}}

\newcommand\dd{{\mathrm d}}

\newcommand\ts{\textstyle}

\newtheorem{thm}{Theorem}[section]

\newtheorem{lemma}[thm]{Lemma}
\numberwithin{equation}{section}

\makeatletter
\def\blfootnote{\xdef\@thefnmark{}\@footnotetext}
\makeatother

\begin{document}

\title[Expander estimates]
{Expander estimates for cubes} 




\author{J\"org Brüdern and Simon L. Rydin Myerson}










\begin{abstract} If $\mathscr A$ is a set of natural numbers of exponential density $\delta$, then the exponential density of all numbers of the form $x^3+a$ with $x\in\mathbb N$ and $a\in\mathscr A$ is at least $\min(1, \f13+\f56 \delta)$. This is a considerable improvement on the previous best lower bounds for this problem, obtained by Davenport more than 80 years ago. The result is the best possible for $\delta\ge \f45$.   \end{abstract}
\blfootnote{Keywords: sumsets, exponential density, differencing, exponential sums.}
\blfootnote{MSC(2020): 11B13, 11P55}
\blfootnote{The first author was supported by Deutsche Forschungsgemeinschaft (DFG) project numbers 255083470 and 462335009. 
The second author was supported by DFG project number 255083470 and a Leverhulme Early Career Fellowship, and received support from the Max Planck Institute for Mathematics and from the European Reseach Council (ERC) under the European Union’s Horizon 2020 research and innovation program (grant ID 648329). For the purpose of open access, the author has applied a Creative Commons Attribution (CC-BY) licence to any Author Accepted Manuscript version arising from this submission.
}

\maketitle

\section{Introduction}
Given an infinite set $\mathscr A$ of natural numbers, and an integer $k\ge 2$, we are interested in the sumset
\begin{equation}\label{11} \mathscr B_k =\{x^k+a: \, x\in\mathbb N,\, a\in\mathscr A\}. 
\end{equation}
In particular, we wish to measure how $\mathscr A$ expands to $\mathscr B_k$. This is a classical question that prominently impacted the additive theory of numbers, and sparked progress with Waring's problem in the 1930s and 1940s. For a quantitative description of the problem,
let $\delta(\mathscr A)$ denote the supremum of all real numbers $t\ge 0$ where
$$  
\limsup_{N\to\infty} N^{-t} \# \{a\in\mathscr A: a\le N\} >0. 
$$
The number $\delta=\delta(\mathscr A)$ is referred to as the exponential density of $\mathscr A$, and the goal is to estimate from below the number
$\delta_k= \delta_k(\mathscr A) = \delta(\mathscr B_k)$. 
Because the $k$-th powers near a large number $X$ are spaced apart roughly by $X^{(k-1)/k}$, one might expect that a sequence $\mathscr A$ with $ \delta(\mathscr A)\le (k-1)/k$ is sufficiently sparse to ensure that there are few repetitions among the numbers $x^k+a$ with $x\in\mathbb N$ and $a\in\mathscr A$. This suggests that perhaps one has
\begin{equation}\label{14} \delta_k = \min\big( \delta + (1/k) , 1\big). \end{equation}

There is a simple and very familiar argument (see \cite[Theorem 6.2]{hlm}) that confirms \eqref{14} for $\delta\le 1/k$. Note that this establishes the case $k=2$ of \eqref{14} in full. We briefly recall the instructive reasoning. In this context it is convenient to write $A=A_N$ for the number of $a\in\mathscr A$, $a\le N$. Fix $\varepsilon>0$. Then there are arbitrarily large values of $N$ where $A\ge N^{\delta-\varepsilon}$. For such $N$, we  consider
$$ r(n,N) = \# \{(a,x)\in\mathscr A\times \mathbb N:\, x^k+a=n,\, a\le N,\, x\le N^{1/k}\}. $$
We observe that
$$ \sum_{n\le 2N} r(n,N) = [N^{1/k}]A $$
and apply  Cauchy's inequality to infer
\begin{equation}\label{15} N^{2/k}A^2 \ll \Big(\sum_{n\le 2N} r(n,N)\Big)^2 \le \Big(\multsum{n\le 2N}{r(n,N)\ge 1} 1 \Big) \sum_{n\le 2N} r(n,N)^2. \end{equation}
An upper bound is now required for the sum on the far right of \eqref{15}, and we momentarily establish 
\begin{equation}\label{16} \sum_{n\le 2N} r(n,N)^2 \ll N^{1/k}A + A^2N^\varepsilon. \end{equation}
The definition of $\delta$ implies $A\ll N^{\delta+\varepsilon}$, whence provided that $\delta\le 1/k$, the right hand side of \eqref{16} is $O(N^{1/k + 2\varepsilon}A) $. By \eqref{15} we finally see that
$$ \multsum{n\le 2N}{r(n,N)\ge 1} 1\gg AN^{1/k-2\varepsilon} \gg N^{\delta+1/k-3\varepsilon}.  $$
Since the $n$ counted on the left hand side are elements of $\mathscr B_k$, this proves $\delta_k\ge \delta + 1/k$  when $\delta\le 1/k$, as claimed. For trivial reasons, the number of $b\in\mathscr B_k$ with $b\le N$ is at most $A_N  N^{1/k}$, and this implies 
$\delta_k\le \delta + 1/k$. This confirms \eqref{14} for $\delta\le 1/k$.

For a proof of \eqref{16}, we first observe that the sum to be bounded equals the number of solutions of 
\begin{equation}\label{161} y^k - x^k = a-b  \end{equation}
in integers $x,y,a,b$ with $1\le x,y\le N^{1/k}$ and $a,b\in\mathscr A$, $a,b\le N$. The number of such solutions with $a=b$ is $[N^{1/k}]A$. Meanwhile, when $a\neq b$, a divisor function estimate shows that the number of solutions of \eqref{161} in $x,y$ is $O(N^\varepsilon)$, and the number of choices for $a$, $b$ is $A^2$. This establishes \eqref{16} at once.

One may also run the circle method on the equation \eqref{161}, at least in an informal way. If one is ready to accept the widely held belief that $k$-th power Weyl sums achieve square-root cancellation on appropriate minor arcs, then one is led to expect that
$$ \sum_{n\le 2N} r(n,N)^2 \ll AN^{1/k+\varepsilon} + A^2N^{(2-k)/k+\varepsilon}.$$ 
If true, one may feed this bound to \eqref{15} to infer that \eqref{14} holds in general. This second heuristical argument in support of \eqref{14} gives us reason to name \eqref{14} the {\em expander conjecture for $k$-th powers}. Despite its simple proof for $k=2$, and for the limited range $\delta\le 1/k$ for larger $k$, no other case of our conjecture is currently known. The best lower bounds for $\delta_k$ are due to Davenport \cite{D}, dating back to 1942 (see also \cite[Theorem 6.2]{hlm}). In the cubic case $k=3$, he showed that if $\frac13 \le \delta <1$, then one has
\begin{equation}\label{17} \delta_3 \ge \frac13 \Big( 1+ \frac{4\delta}{1+\delta}\Big). \end{equation}
Note here that this falls short of proving \eqref{14}. In particular, for all $\delta<1$ we fail to show that $\delta_3=1$. 

In this paper we concentrate on the cubic case. We improve Davenport's lower bound \eqref{17} for $\delta > \frac35$, and for $\delta\ge \frac45$ we now reach the best possible result $\delta_3=1$. 

\begin{thm} \label{exp} Let $\mathscr A\subset \mathbb N$ be a set of exponential density $\delta$. If $\delta\le \f45$, then $\delta_3\ge \frac13 + \frac56\delta$. If $\delta \ge \frac45$, then $\delta_3=1$.
\end{thm}  

Our approach to proving this result is designed for application to cubes, but extends to larger values of $k$ in principle. However, already for moderately sized $k$, a more direct approach should perform better. Rather than following the path proposed here, one may apply recent bounds for Weyl sums \cite{BDG,TDW} that become available through our now almost perfect control of Vinogradov's mean value, and one should be able to show that when $1-\delta \le 2k^{-3}$, then $\delta_k=1$. For comparison, as in the cubic case, Davenport's estimates have the appalling feature that $\delta_k=1$ is achieved only when one already has $\delta=1$.

Davenport used his lower bounds for $\delta_k$ in his work on Waring's problem. With such applications in mind, various methods have been developed that improve on Davenport's estimates for specific sets $\mathscr A$. In Waring's problem, one may explore the homogeneity of sums of $k$-th powers. Davenport \cite[Theorem 4]{D} proposed one such method that was later developed by Vaughan \cite{V86,V89}. Early experiments in the applicability of the circle method  are due to Vaughan \cite{V85} and Thanigasalam \cite{T}. In Vaughan's work it was important that the input set $\mathscr A$ stems from a polynomial, or at the very least, that $\mathscr A$ is somewhat regularly distributed over arithmetic progressions. He showed that sums of three positive cubes have exponential density at least $\frac89$. We take $\mathscr C = \{x^3+y^3: \, x,y\in\mathbb N\}$, and note that $\delta(\mathscr C)=\frac23$. Our Theorem \ref{exp} yields  $\delta_3(\mathscr C)\ge\frac89$ and recovers Vaughan's result\footnote{For the best currently known lower bound for $\delta_3(\mathscr C)$ see \cite{WW}.}. This is not a coincidence. In fact, our Theorem \ref{exp} results from an attempt to trim Vaughan's argument to apply to general input sets. 

The simple argument that gave \eqref{14} is rather flexible. We modify it following Davenport \cite{D} and Vaughan \cite{V85}. Choose $P=N^{2/5}$ and then examine the number of solutions of \eqref{161} with $k=3$ and
$P<x,y\le 2P$. Note that implicit in the equation and the size constraints is the condition $|x-y|\le\surd P$. Like Vaughan, we apply the circle method to the equation \eqref{161},  bring in $h=y-x$ as a new variable, and are then led to study the exponential sum
\begin{equation}\label{W} F_1(\alpha) = \sum_{1\le h\le H} \sum_{P<x\le 2P}
e(\alpha  h(3x^2+3xh +h^2)), \end{equation}
where $H=\surd P$.
Vaughan \cite{V85} applies Weyl differencing. This bounds the sum \eqref{W} by $O(P^{1+\varepsilon})$ on minor arcs. On major arcs, the variable $\alpha$ is close to a fraction $a/q$ with $q\le P$ and $(a,q)=1$. In this situation Weyl differencing  yields the upper bound $O(P^{3/2}q^{-1/2})$, too weak to compute the major arc contribution satisfactorily. Indeed, one seems to require an upper bound on major arcs that is not much weaker than $P^{3/2}q^{-1}$. As an alternative approach, one may apply Poisson summation to the double sum \eqref{W}. The main term then factors into local parts, and at finite places the square of the modulus  of a classical cubic Gauss sum occurs. This observation is due to Zhao \cite{Z}. For cube-free $q$, this achieves the desired upper bound $P^{3/2}/q$. Unfortunately, when $q=r^3$ is a pure cube, the actual size of \eqref{W} is about $P^{3/2}r^{-2}$. This is disappointing, and seems to cast doubt on the potential of this approach. However, it can be shown that the peaks of \eqref{W} at $\alpha=a/r^3$ arise from pairs $x,y=x+h$ with $x$ and $y$ divisible by $r$, and that the remaining part of the sum is of favourable size. Thus, if at least one of $x$ or $y$ were to run over integers coprime to $r$ only, then the offensive effects from cubic moduli $r^3$ would be suppressed. We achieve this by initially working with cubes of primes. Once the argument takes us to equation \eqref{161} (with $k=3$), we cover the primes with integers $x,y$ satisfying $(x,y,w)=1$, with the parameter $w$ free of prime factors exceeding $P$, but otherwise at our disposal. Thus, we now work with the more general sum
\begin{equation}\label{33} F_w(\alpha) = \sum_{1\le h\le H} \multsum{P<x\le 2P}{(x,h,w)=1}
e(\alpha  h(3x^2+3xh +h^2)). \end{equation}
As it turns out, the condition  $(x,y,w)=1$ eliminates unwanted peaks of $\eqref{W}$ at $a/q$ as long as the cubic factors of $q$ are solely composed of primes dividing $w$. In  order to describe this precisely, we need to introduce more notation. The estimate 
for $F_w$ features prominently a certain multiplicative function $\kappa_w$ that we define
on powers of a prime $p$ with $p\nmid w$ by
$$ \kappa_w(p) = 2p^{-1/2},\quad \kappa_w(p^2)= p^{-1}, $$
while for all $l\ge 0$ we put $\kappa_w(p^{l+3}) = p^{-1}\kappa_w(p^l)$. When $p\mid w$  and $p\neq 3$, we put
$$ \kappa_w(p) = 2p^{-1/2},\quad \kappa_w(p^l)= 0 \quad (l\ge 2),$$
while when $3\mid w$, we take
$$ \kappa_w(3)= 2, \quad \kappa_w(9) = 1, \quad  \kappa_w(3^l)= 0 \quad (l\ge 3).$$
For later use, we note at once that one has $\kappa_w(p^l) \le p^{-l/3}$ for all $p\ge 67$ and all $l\in\NN$. For $p\le 61$ and again all $l\in\NN$, one only has $\kappa_w(p^l) \le 2p^{-l/3}$.
This implies  the inequality
\begin{equation}\label{kapp} \kappa_w(q) \le 2^{18} q^{-1/3}\quad (q\in\NN, w\in\NN). \end{equation}

We also require a Farey dissection of the unit interval.
Given integers $a,q$ with $0\le a\le q\le P$ and $(a,q)=1$, let
$$ \mathfrak M(q,a)= \{\alpha\in[0,1]:\, |q\alpha-a|\le (6HP)^{-1}\}, $$
and for $\alpha\in\mathfrak M(q,a)$ put
$$ \Ups_w(\alpha) = \kappa_w(q)^2 (1+HP^2|\alpha-a/q|)^{-1}.$$
The intervals $\mathfrak M(q,a)$ are disjoint, and we denote their union by $\mathfrak M$. We also write $\mathfrak m= [0,1]\setminus \mathfrak M$ and put $\Ups_w(\alpha)=0$ for $\alpha\in\mathfrak m$. This defines a function $\Ups_w: [0,1]\to [0,1]$. 

We are now in a position to announce the principal result that we expect to be useful well beyond the problem considered in this paper. 

\begin{thm}\label{thm31} Uniformly in $w\in\mathbb N$ and $\alpha\in[0,1]$, one has
$$ F_w(\alpha) \ll HP(\log P)\Ups_w(\alpha) + P^{1+\varepsilon}. $$
\end{thm}

We only require a special case of Theorem \ref{thm31} where it is possible to sharpen the conclusion somewhat. We put $Q=P^{3/4}$.
Given integers $a,q$ with $0\le a\le q\le Q$ and $(a,q)=1$, let
$$ \mathfrak N(q,a)= \{\alpha\in[0,1]:\, |q\alpha-a|\le P^{-7/4}\}, $$
and for $\alpha\in\mathfrak N(q,a)$ put
$$ \Xi(\alpha) = 4^{\omega(q)}(q+HP^2|q\alpha-a|)^{-1},$$
where $\omega(q)$ denotes the number of distinct prime factors of $q$.
Note that $\mathfrak N(q,a)\subset\mathfrak M(q,a)$. Hence, 
the intervals $\mathfrak N(q,a)$ are again disjoint, and we denote their union by $\mathfrak N$. For $\alpha\in [0,1]\setminus \mathfrak N$ we put $\Xi(\alpha)=0$ to define a function $\Xi: [0,1]\to [0,1]$.

\begin{thm}\label{thm32} Let $\varpi$ denote the product of all primes not exceeding $P^{1/4}$. Then 
$$ F_\varpi(\alpha) \ll HP(\log P) \Xi(\alpha) + P^{1+\varepsilon}. $$
\end{thm}

With Theorem \ref{thm31} in hand, the proof of its corollary, named here Theorem \ref{thm32}, is so easy that we give it at once. In fact, by \eqref{kapp}, we see that whenever $\alpha\in\mathfrak M\setminus \mathfrak N$, then $\Ups_w(\alpha) \ll P^{-1/2}$ holds
 uniformly in $w$. This shows that one has $F_w(\alpha)\ll P^{1+\varepsilon}$ for these $\alpha$. Further, when $\alpha\in \mathfrak N(q,a)$ for some coprime pair $a,q$ with $q\le Q$ and $q$ has a cubic factor, then in particular, there is a prime $p$ with $p^3\mid q$. Now $p\le q^{1/3}\le P^{1/4}$, and hence $\kappa_\varpi(q)=0$, where $\varpi$ is as in Theorem \ref{thm32}. Consequently, in this case we again have 
$F_\varpi(\alpha)\ll  P^{1+\varepsilon}$. This leaves the case where $q$ is cube-free. Then the inequality $\kappa_\varpi(q) \le 2^{\omega(q)}q^{-1/2} $ is immediate from the definition of $\kappa_w$, and the conclusion of Theorem \ref{thm32} follows from Theorem \ref{thm31}.

Unfortunately, the coprimality constraint present in \eqref{33} interferes with the previously rather straightforward treatments of major and minor arcs for the sum $F_1(\alpha)$, and causes a larger number of technical complications. This forces us to first review quadratic Weyl sums in Section 2. Theorem 2.1 below refines important work of Vaughan \cite{gen}, and may be of independent interest. In Sections 3 to 5 we estimate $F_w(\alpha)$ on the major arcs. We draw inspiration from an unweighted version \cite{B24} of Zhao's work \cite{Z}.  Although technically demanding, our arsenal here solely contains materiel that has been in use for some time. Our treatment of the minor arcs in Sections 6 and 7 uses an inequality of Weyl's type, but then also invokes the Poisson summation formula, and finally depends on a succession of counting processes that do not seem to have any ancestors in the literature.

In this paper, we apply the following convention concerning the letter $\varepsilon$. Whenever $\varepsilon$ occurs in a statement, it is asserted that the statement is true for any positive real value assigned to $\varepsilon$. Constants implicit in the use of Vinogradov's and Landau's familiar symbols may depend on the value of $\varepsilon$. We abbreviate $\exp(2\pi \mathrm i \alpha)$ to $e(\alpha)$. The number of divisors of the natural number $q$ is denoted $\tau(q)$,  and $\phi(q)$ is Euler's totient function. {We reserve the letter $d$ to denote a square-free number; we may also use $d_0,d'$ and so forth. As our main objects of study, $F_w$ and $\kappa_w$, depend on $w$ only on through its square-free part, we will invariably view $w$ through the lens of its square-free divisors $d$.}

An anonymous referee proposed a substantial number of minor alterations to an earlier version of this text. These were evidently thought through with care, and we feel the result is quite a bit more readable. We gratefully acknowledge this invaluable help.

\section{The quadratic Weyl sum}

In this section, we study the general quadratic Weyl sum
\begin{equation}\label{21} f(\alpha_1,\alpha_2) = \sum_{x\le X} e(\alpha_1 x+ \alpha_2 x^2). \end{equation}
The local analogues
\begin{equation}\label{22} S(q,a_1,a_2) = \sum_{x=1}^q e\Big(
\frac{a_1x+a_2 x^2}{q}\Big) \end{equation}
and
$$ I(\beta_1,\beta_2) = \int_0^X e(\beta_1 u+ \beta_2 u^2) \dd u $$
are featured in our main result. 

\begin{thm}\label{quadWeyl} Let $a_1,a_2\in\mathbb Z$, $q\in\mathbb N$, $\alpha_j\in\mathbb R$ and $\beta_j = \alpha_j-a_j/q$. Suppose that $|\beta_1|\le 1/(2q)$. Then
$$ f(\alpha_1,\alpha_2) - q^{-1}S(q,a_1,a_2)I(\beta_1,\beta_2) \ll (q,a_2)^{1/2} (q+qX^2|\beta_2|)^{1/2} \log q. $$
\end{thm}

The result is similar to \cite[Theorem 8]{gen} where it is required that $(a_2,q)=1$. The factor $\log q$ is absent in \cite{gen}. This factor turns out to do no harm to our applications; it is more relevant to remove the coprimality constraint on $a_2$. We give a detailed proof.

\begin{lemma}\label{L22} Let $a_1,a_2\in\mathbb Z$ and  $q\in\mathbb N$. Then
$$ \lim_{X\to\infty} X^{-1} f(a_1/q,a_2/q) = q^{-1}S(q,a_1,a_2). $$
\end{lemma}  

\begin{proof} We sort the sum in \eqref{21} into residue classes modulo $q$. Then
$$ f\Big(\frac{a_1}q, \frac{a_2}q\Big) = \sum_{b=1}^q \multsum{x\le X}{x\equiv b\bmod q} e\Big(\frac{a_1 b+a_2b^2}q\Big) =   \sum_{b=1}^q e\Big(\frac{a_1 b+a_2b^2}q\Big) \Big(\frac{X}{q} +O(1)\Big), $$
and the result follows from \eqref{22}.\end{proof}

Given $a_1,a_2,q$, we may cancel $(a_1,a_2,q)=d$ and put $q' = q/d$, $a'_j=a_j/d$. Then $(a'_1,a'_2,q')=1$, and $f(a_1/q,a_2/q) = f(a'_1/q',a'_2/q')$. Two applications of Lemma \ref{L22} now show that
\begin{equation}\label{24} q^{-1}S(q,a_1,a_2)= q^{\prime -1} S(q',a'_1,a'_2). \end{equation}

\begin{lemma}\label{L23} If $(a_1,a_2,q)=1$, then $S(q,a_1,a_2)\ll q^{1/2}$. If in addition $(q,a_2)>1$, then $S(q,a_1,a_2)=0$. \end{lemma}

\begin{proof}
We establish the second clause first. Let $p$ be a prime with $p\mid (q,a_2)$, and define $s$ and $t$ via $p^t\parallel  q$, $p^s\parallel (q,a_2)$. Then $1\le s\le t$. We now apply the quasi-multiplicative properties of $S(q,a_1,a_2)$ described in \cite[Lemma 4.1]{hlm}. In this way we see that $S(q,a_1,a_2)$ has a factor
$S(p^t,b_1,p^sb_2)$ where the integers $b_1,b_2$ are not divisible by $p$. Note here that $p\mid a_2$ gives $p\nmid a_1$, and hence $p\nmid b_1$. Now, in
$$ S(p^t,b_1,p^sb_2) = \sum_{x=1}^{p^t} e\Big( \frac{b_1x+p^sb_2 x^2}{p^t}\Big)$$
we put $x=p^{t-s}y+z$ and get
$$ S(p^t,b_1,p^sb_2) = \sum_{z=1}^{p^{t-s}} \sum_{y=1}^{p^s} e\Big(\frac{b_1y}{p^s}\Big) e\Big(\frac{b_1 z}{p^t} + \frac{b_2z^2}{p^{t-s}}\Big) = 0. $$
This proves the second clause. With this in hand, it suffices to establish the first clause in the case where $(a_2,q)=1$, and this situation is fully covered by \cite[Lemma 2.1]{gen}. 
\end{proof}

\begin{lemma} \label{L24}Let $a_1,a_2\in\mathbb Z$ and $q\in\mathbb N$. Then
$S(q,a_1,a_2) \ll q^{1/2} (q,a_1,a_2)^{1/2}. $
\end{lemma} 

\begin{proof}
Remove common factors in $a_1,a_2,q$ by \eqref{24} and apply the estimate provided by Lemma \ref{L23}.
\end{proof}

We are ready to embark on the proof of Theorem \ref{quadWeyl}. In the notation introduced in the statement of this theorem, we infer from \eqref{21}, \eqref{22} and orthogonality that
$$ f(\alpha_1,\alpha_2) = \frac1{q} \sum_{b=1}^q S(q,a_1+b,a_2)f\Big(\beta_1-\frac{b}{q}, \beta_2\Big). $$
We may now apply \cite[Lemma 2.3]{gen} to the effect that
$$ f(\alpha_1,\alpha_2) =
\frac1{q} \sum_{b=1}^q S(q,a_1+b,a_2)
\Bigg(\!\multsum{m\in\mathbb Z}{|m+(b/q) -\beta_1|\le 1+4|\beta_2|X}\!\! I\Big(\beta_1 -\frac{b}{q} -m, \beta_2\Big) + O(1)\! \Bigg). $$
By Lemma \ref{L24}, the right hand side here  simplifies to
$$
\frac1{q} \sum_{b=1}^q S(q,a_1+b,a_2)
\multsum{m\in\mathbb Z}{|m+(b/q) -\beta_1|\le 1+4|\beta_2|X}\!\! I\Big(\beta_1 -\frac{b}{q} -m, \beta_2\Big) + O(q^{1/2}(q,a_2)^{1/2}).$$
The transformation $n=mq+b$ allows us to recast the last estimate as
\begin{multline}\label{26}f(\alpha_1,\alpha_2) =\\
\frac1{q}\multsum{n\in\mathbb Z}{|(n/q) -\beta_1|\le 1+4|\beta_2|X} S(q,a_1+n,a_2)I\Big(\beta_1 -\frac{n}{q}, \beta_2\Big) + O(q^{1/2}(q,a_2)^{1/2}).\end{multline}
Recall that $|\beta_1|\le 1/(2q)$. Hence,  the interval for $n$ in the summation condition on the right hand side contains $n=0$, and this corresponds to $q^{-1}S(q,a_1,a_2)I(\beta_1,\beta_2)$, the term that is to be subtracted from $f(\alpha_1,\alpha_2)$. It  remains to show that the contribution of terms with $n\neq0$ is small. 

We temporarily suppose that $|\beta_2|\le (5qX)^{-1}$. With this extra assumption, it is shown on p.\ 445 of \cite{gen} that whenever $n\neq 0$, one has
$$ I\Big(\beta_1 -\frac{n}{q}, \beta_2\Big) = \frac{e(-un/q)}{-2\pi\mathrm i n/q} e(\beta_1 u+\beta_2u^2) \Bigg|_{u=0}^X + O\Big(\frac{q}{n^2}\Big). $$
This may be implanted into \eqref{26}. The error terms $O(q/n^2)$ generate an error
in \eqref{26} that in light of Lemma \ref{L24} does not exceed
$$ \ll \sum_{n\neq 0} \frac{|S(q,a_1+n, a_2)|}{n^2} \ll q^{1/2} (q,a_2)^{1/2}. $$
We now have
$$ f(\alpha_1,\alpha_2)-q^{-1}S(q,a_1,a_2)I(\beta_1,\beta_2) = R+O(q^{1/2} (q,a_2)^{1/2}) $$
where
\begin{equation}\label{27} R= \multsum{n\neq 0}{|(n/q) -\beta_1|\le 1+4|\beta_2|X} S(q,a_1+n,a_2)
\frac{e(-un/q)}{-2\pi\mathrm i n} e(\beta_1 u+\beta_2u^2) \Bigg|_{u=0}^X. \end{equation}
According to current hypotheses, we have $q|\beta_1|\le \frac12$, so the range 
$0<|n| < q$ is always included in the summation condition in   \eqref{27}. However, we also have $|\beta_2| \le (5qX)^{-1}$, so if $ |(n/q) -\beta_1|\le 1+4|\beta_2|X$,
then $|n|\le q+2$. In the expression for $R$ we keep the terms with $0<|n| < q$, and estimate the remaining summands, of which there may be at most 6, trivially, only using that $|n|\ge q$ now. This gives
$$ R= \sum_{0<|n|<q} S(q,a_1+n,a_2)
\frac{e(-un/q)}{-2\pi\mathrm i n} e(\beta_1 u+\beta_2u^2) \Bigg|_{u=0}^X +O(1). $$
We now find via Lemma \ref{L24} that
$$ R \ll \sum_{0<|n|<q} \frac{|S(q,a_1+n,a_2)|}{|n|} \ll q^{1/2}(q,a_2)^{1/2}
\sum_{1\le n\le q} \frac1n . $$
Subject to the additional hypothesis $|\beta_2|\le (5qX)^{-1}$, the proof of the theorem is now complete.

We may now suppose that $|\beta_2|> (5qX)^{-1}$. In this case, the conclusion of the theorem is an easy consequence of \cite[Theorem 4]{gen}. To see this, write
$a_2/q$ in lowest terms, say $a_2/q=a_0/q_0$. Then $q_0=q/(q,a_2)$, and the aforementioned Theorem 4 of Vaughan \cite{gen} gives
\begin{align*} f(\alpha_1,\alpha_2)& \ll X(q_0 +X^2|q_0\alpha_2-a_0|)^{-1/2}
+ (q_0+X^2|q_0\alpha_2-a_0|)^{1/2}\\
& \ll Xq_0^{-1/2} (1+X^2|\beta_2|)^{-1/2} + q_0^{1/2} (1+X^2|\beta_2|)^{1/2}.
\end{align*}
Now from $|\beta_2|> (5qX)^{-1}$ we see that $|q\beta_2|^{-1/2} \ll (qX^2|\beta_2|)^{1/2}$, and so the previous estimate for $f(\alpha_1,\alpha_2)$ reduces to
$$ f(\alpha_1,\alpha_2) \ll q^{1/2}(q,a_2)^{1/2} (1+X^2|\beta_2|)^{1/2}. $$
To complete the argument, we invoke the bound $I(\beta_1,\beta_2)\ll |\beta_2|^{-1/2}$ that is valid for all $\beta_1,\beta_2$ with $\beta_2\neq 0$ (see \cite[Lemma 2.2]{gen}). But then again, by Lemma \ref{L24},
$$ q^{-1}S(q,a_1,a_2)I(\beta_1,\beta_2) \ll (q,a_2)^{1/2} q^{-1/2} |\beta_2|^{-1/2}
\ll (qX^2|\beta_2|)^{1/2} (q,a_2)^{1/2}, $$
and the case $|\beta_2|> (5qX)^{-1}$ is now established, too. 

\section{Major arcs: first application of Poisson summation}
From now on, our main parameter is $P$, and we continue to write
$$ H= \sqrt P. $$
We consider the exponential sum 
\begin{equation}\label{35} G(\alpha)=G(\alpha;X,Y)= \sum_{h\le Y}\sum_{X<x\le 2X} e(\alpha  h(3x^2+3xh +h^2))\end{equation}
that relates to \eqref{33} via 
 M\"obius inversion. This results into the identity
\begin{equation}\label{36} F_w(\alpha) = \sum_{d\mid w} \mu(d) G(\alpha d^3; P/d,H/d). \end{equation}
Note here that for $d>H$ one has $G(\beta,P/d,H/d)=0$ so that implicitly the sum over $d$ is constrained to $d\le H$. We shall invoke this comment frequently and silently. We require all our estimates uniformly relative to the auxiliary natural number $w$. 
Note here that primes $p$ with $p>H$ are coprime to all $h\le H$. Hence, there is no loss of generality in assuming henceforth that $w$ is a number free of prime factors exceeding $H$.

The goal of this section is to evaluate the quadratic Weyl sum hidden as the sum over $x$ in \eqref{35}, by Poisson summation. We are able to import this step through the work in Section 2.  
In the remainder of this and the next two sections, we suppose that $\alpha\in\mathfrak M$. Then $\alpha\in [0,1]$, and there are coprime integers $a$ and $q$ with
\begin{equation}\label{maj} 0\le a\le q\le P \quad \text{ and } \quad\alpha = a/q +\beta, \quad q|\beta|\le (6HP)^{-1}. \end{equation}
We use this notation throughout.

As a first step, we rearrange $F_w(\alpha)$ to prepare for an application of 
Theorem \ref{quadWeyl}. It will be convenient to write
\begin{equation}\label{37} g(\alpha_1,\alpha_2;X) = \sum_{X<x\le 2X} e(\alpha_1 x+\alpha_2 x^2), \quad J(\beta_1,\beta_2;X) = \int_X^{2X} e(\beta_1 u+\beta_2 u^2)\,\dd u, \end{equation}
and we abbreviate $g(\alpha_1,\alpha_2;P)$ to $g(\alpha_1,\alpha_2)$, and 
$J(\beta_1,\beta_2;P)$ to $J(\beta_1,\beta_2)$. By \eqref{35} and \eqref{36}, we have
\begin{equation}\label{38} F_w(\alpha) = \sum_{d\mid w} \mu(d) \sum_{h\le H/d} e(\alpha d^3h^3)
g(3\alpha d^3h^2, 3\alpha d^3h; P/d). \end{equation}
Recall the Gauss sum $S$ introduced in \eqref{22}.

\begin{lemma}\label{L34} Let $\alpha\in\mathfrak M$  and suppose that \eqref{maj} applies. Then, uniformly in $w$, one has
$$ F_w(\alpha) = \sum_{d\mid w}\mu(d) \sum_{h\le H/d} e(\alpha d^3 h^3)
\frac{S(q,3ad^3h^2,3ad^3h)}{qd} J(3h^2d^2\beta, 3hd\beta) + O(P^{1+\varepsilon}).$$ 
\end{lemma}
 
\begin{proof}
For $j=1$ and $2$ we have
$$ 3\alpha d^3 h^j = \frac{3ad^3h^j}{q} + 3h^jd^3\beta. $$
Here we cancel  common factors between $q$ and $d^3$. On writing $q_0=q/(q,d^3)$ and $d_0=d^3/(q,d^3)$, we have $(q_0,d_0)=1$ and
$$    3\alpha d^3 h^j = \frac{3ah^jd_0}{q_0} + 3h^jd^3\beta \quad (j=1,2). $$
We wish to apply Theorem \ref{quadWeyl} with $X=2P/d$ and $X=P/d$ to evaluate the sum $g(3\alpha d^3h^2, 3\alpha d^3h; P/d)$. We should then choose 
$$\alpha_1= 
3\alpha d^3 h^2, \quad \alpha_2 =3\alpha d^3 h, \quad \beta_1=3 d^3 h^2 \beta\quad \beta_2 = 3 d^3 h\beta. $$
Recalling that $hd\le H$, we see that $|\beta_1|\le 3H^3/(6HPq)=1/(2q)$. Theorem \ref{quadWeyl} is therefore applicable in the way we desired, and we infer that
\begin{equation}\label{39} 
  g(3\alpha d^3h^2, 3\alpha d^3h; P/d) = q_0^{-1} S(q_0, 3ah^2d_0, 3ahd_0) J(3 d^3 h^2 \beta,3 d^3 h\beta;P/d) + E \end{equation}
where
$$ E \ll (q_0,h)^{1/2} (q_0 +q_0 P^2dh |\beta|)^{1/2} \log q. $$
But $q_0\mid q$,  $q \le P$ and $|\beta|\le 1/(6HPq)$ simplify this bound to
$ E\ll (q,h)^{1/2} P^{1/2+\varepsilon}$. If one inserts the expansion \eqref{39} to \eqref{38}, the errors $E$ sum to an amount no larger than
$$ P^{1/2+\varepsilon} \sum_{d\le H} \sum _{h\le H/d} (q,h)^{1/2} \ll P^{1+2\varepsilon}. $$
We have now proved that
\begin{multline*}F_w(\alpha) = \\
\sum_{d\mid w}\mu(d) \sum_{h\le H/d} e(\alpha d^3 h^3) \f{S(q_0, 3ah^2d_0, 3ahd_0)}{q_0}J(3 d^3 h^2 \beta,3 d^3 h\beta;P/d) +O(P^{1+\varepsilon}).\end{multline*}
By Lemma \ref{L23}, we have 
$$ q_0^{-1} S(q_0, 3ah^2d_0, 3ahd_0)= q^{-1} S(q,3ad^3h^2, 3ad^3h). $$
Further, the substitution $u=v/d$ in \eqref{37} gives 
$$ J(3 d^3 h^2 \beta,3 d^3 h\beta;P/d) = d^{-1} J(3h^2d^2\beta, 3hd\beta). $$
Collecting all this together, one arrives at the claim of the lemma.
\end{proof} 

Further progress now depends on an upper bound for the integral $J(\beta_1,\beta_2)$.

\begin{lemma}\label{LJ} Let $\beta_1\ge 0$, $\beta_2>0$. Then $J(\beta_1,\beta_2)\ll 1/(P\beta_2)$.
\end{lemma}

\begin{proof}
We rewrite the integral in \eqref{37} as
$$ J(\beta_1,\beta_2) = \int_P^{2P} \frac{1}{2\pi\mathrm i (\beta_1+2\beta_2 u)} \,\frac{\mathrm d}{\mathrm d u} \, e(\beta_1 u+\beta_2 u^2)\dd u. $$
This is possible because $\beta_1+2\beta_2 u \ge \beta_2 P$ holds for all $u\ge P$. Integration by parts confirms the claim of the lemma.
\end{proof}

Of course, one also has the trivial bound $|J(\beta_1,\beta_2)|\le P$. 
As an immediate corollary of Lemma \ref{LJ}, we have the bound
\begin{equation}\label{310} J(3h^2d^2\beta, 3hd\beta) \ll P(1+P^2hd|\beta|)^{-1}  \end{equation}
for all $\beta\in \mathbb R$. Indeed, this follows from Lemma \ref{LJ} combined with the trivial bound when $\beta \ge 0$, and when 
$\beta<0$ one observes that 
$J(-\gamma,-\theta)= \overline{J(\gamma,\theta)}$ to reduce to the case where $\beta>0$.

Equipped with the estimate \eqref{310}, it is possible to truncate the sum over the divisors of $w$ in the formula for $F_w(\alpha)$ that
Lemma \ref{L34} provides. When $\alpha\in\mathfrak M$ and \eqref{maj} applies, define the set
\begin{equation}\label{313} D=D(\alpha, w)= \{d\in\NN: \mu(d)^2=1, \,d\mid w, \, d\le H,\, d|\beta|\le (48P^2)^{-1}\} \end{equation}
and the sum
\begin{equation}\label{314-} F^*_w(\alpha) =  
\sum_{d\in D} \frac{\mu(d)}{qd} \sum_{h\le H/d} e(\alpha d^3 h^3)
S(q,3ad^3h^2,3ad^3h) J(3h^2d^2\beta, 3hd\beta).
\end{equation}

\begin{lemma}\label{L36} Let $\alpha\in\mathfrak M$, and suppose that \eqref{maj} applies. Then, uniformly in $w$, one has $F_w(\alpha)=F^*_w(\alpha) + O(P^{1+\varepsilon})$.
\end{lemma}

\begin{proof} The starting point is the evaluation of $F_w(\alpha)$ provided by Lemma \ref{L34}. We have to show that the contribution to the part of the sum displayed in Lemma \ref{L34} that arises from summands where  $d\mid w$ and $d|\beta|> (48P^2)^{-1}$ is negligible. For these $d$, one applies \eqref{310} to see that
$$ J(3h^2d^2\beta, 3hd\beta) \ll P(1+h)^{-1} \ll P/h. $$
The trivial bound $|S(q,a,b)|\le q$ is enough to conclude that the undesired divisors $d$ contribute to the right hand side of the expression for $F_w$ in Lemma \ref{L34} at most
$$  P\sum_{d|w} \sum_{h\le H/d} \frac{1}{hd} \le P \sum_{dh\le H} \frac{1}{hd} \ll P (\log P)^2, $$
as required.  
\end{proof} 

\section{Major arcs: second application of Poisson summation}
In this section we evaluate the sum over $h$ in \eqref{314-}. Again, our principal tool is a truncated version of the Poisson summation formula. We shall encounter the Fourier integral
\begin{equation}\label{321} 
 K(\beta) = \int_0^H\!\!\!\int_P^{2P} e(\beta(v^3+3uv^2+3u^2v))\,\mathrm du\,\mathrm dv \end{equation}
and Hua's exponential sum
\begin{equation}\label{322} U(q,a,b) = \sum_{z=1}^q e((az^3+bz)/q).\end{equation}
In the interest of brevity, we write $U(q,a,0)=U(q,a)$ for the cubic Gauss sum. {We continue to denote by $D$ the particular set of square-free numbers defined in \eqref{313}.}

\begin{lemma}\label{L37} Let $\alpha\in\mathfrak M$. Then, uniformly in $w$, one has
$$ F^*_w(\alpha) = \sum_{d\in D(\alpha,w)} \frac{\mu(d)}{d^2q^2 } |U(q,ad^3)|^2 K(\beta) + O(P^{1+\varepsilon}). $$ 
\end{lemma}

The proof is somewhat involved. We require some auxiliary estimates that we present first.

\begin{lemma}\label{Huasum} Let $a,b,q,d$ be natural numbers with $(a,q)=1$ and $d$ square-free. Then
$$ U(q,ad^3,b) \ll q^{1/2+\varepsilon} (q,b)^{1/2}. $$
\end{lemma} 

\begin{proof} We begin with the identity
\begin{equation}\label{H0} U(q_1q_2,c,b) = U(q_1,cq_2^2,b) U(q_2,cq_1^2,b) \end{equation}
that holds for coprime natural numbers $q_1,q_2$ and any integers $b,c$. This is  proved {\em en passant} on p.\ 38 of Vaughan \cite{hlm}, subject to the condition that $(c,q_1 q_2)=1$. However, an inspection of the underlying argument shows that this extra condition is not required.

Now suppose that $c$ is an integer with $(c,q)=1$. Then, it follows from Hooley \cite[eqn.\ (43)]{Hacta} that 
\begin{equation}\label{H1} U(q,c,b) \ll C^{\omega(q)} q^{1/2} (q,b)^{1/4}. \end{equation}
Here $C>1$ is a suitable absolute constant. In particular, this proves the lemma in the special case where $d$ and $q$ are coprime.

Let $q_0$ denote the largest divisor of $q$ with $(q_0,d)=1$, and write $q=q_0 q_1$. Then $(q_0,q_1)=1$, and \eqref{H0} gives 
$$ U(q,ad^3,b) = U(q_0, ad^3q_1^2,b) U(q_1,ad^3q_0^2,b). $$
We shall now prove, by induction on the number of distinct prime factors of $q_1$, that whenever $(c,q_1)=1$, one has
\begin{equation}\label{H4} |U(q_1,cd^3,b)|\le C^{\omega(q_1)}q_1^{1/2}(q_1,b)^{1/2}. \end{equation}
Since we may apply \eqref{H1} to $U(q_0,ad^3q_1^2,b)$, we see that the claim of the lemma follows from \eqref{H4}.

If $q_1$ has no prime factor, then $q_1=1$ and \eqref{H4} holds. Thus we may suppose that $q_1$ has at least one prime factor $p$, say. Then $d=pd'$ with $p\nmid d'$. There is a number $\nu\ge 1$ with $q_1=p^\nu q_2$ and $p\nmid q_2$. We suppose that \eqref{H4} is true for $q_2$ in place of $q_1$ and apply \eqref{H0} with $p^\nu$ in the role of $q_1$. We then encounter a factor
$ U(p^\nu, cp^3,b)$ for some suitable $c$ not divisible by $p$ while the induction hypothesis applies to the complementary factor. 
If $1\le \nu\le 3$, then
$$ U(p^\nu, cp^3, b) = \sum_{x=1}^{p^\nu} e\Big(\frac{bx}{p^\nu}\Big). $$
If $p^\nu\mid b$, then 
$$ U(p^\nu, cp^3, b) = p^\nu = p^{\nu/2} (p^\nu,b)^{1/2}. $$
If $p^\nu\nmid b$, then $U(p^\nu, cp^3,b) =0$. Hence, when $\nu\le 3$, we certainly have
$$ |U(p^\nu, cp^3, b)|\le p^{\nu/2}(p^\nu,b)^{1/2}. $$
If $\nu>3$, then the substitution $x= y+p^{\nu-3}z$ in
$$ U(p^\nu,cp^3,b) = \sum_{x=1}^{p^\nu} e\Big(\frac{cx^3}{p^{\nu-3}} +\frac{bx}{p^\nu}\Big) $$
gives
$$ U(p^\nu,cp^3,b) =
\sum_{y=1}^{p^{\nu-3}} \sum_{z=1}^{p^3} e\Big(\frac{cy^3}{p^{\nu-3}} + \frac{by}{p^\nu}\Big) e\Big(\frac{bz}{p^3}\Big). $$
Here the sum over $z$ is zero unless $p^3\mid b$, and in the latter case we conclude that
$$  U(p^\nu,cp^3,b) = p^3 U(p^{\nu-3}, c,b/p^3). $$
We may now apply the bound \eqref{H1} to conclude that
$$  |U(p^\nu,cp^3,b)| \le C p^3 p^{\nu/2 -3/2} (p^{\nu-3}, bp^{-3})^{1/2} \le C p^{\nu/2}(p^\nu,b)^{1/2}. $$
This completes the induction and the proof of Lemma \ref{Huasum}.  
\end{proof}

We also require an estimate for certain oscillatory integrals. The following simple bound is taken from Titchmarsh \cite[Lemma 4.2]{Tit}.

\begin{lemma} \label{Tit} Let $A,B$ be real numbers with $A<B$. Let $\Theta: [A,B] \to \RR$ be  differentiable, and suppose that $\Theta'$ is monotonic on (A,B). Suppose that for all $t$ with $A<t<B$ one has $|\Theta'(t)| \ge \Delta>0$. Then
$$ \Big|\int_A^B e(\Theta(t))\,\mathrm d t\Big| \le \frac1\Delta. $$
\end{lemma} 

We now embark on the proof of Lemma \ref{L37}. The point of departure is \eqref{314-}
where we open $S$ and $J$ via \eqref{22} and \eqref{37}. In the notation fixed in \eqref{maj}, this produces
\begin{equation}\label{314} F^*_w(\alpha) =  
\sum_{d\in D} \frac{\mu(d)}{qd} \sum_{z=1}^q \int_P^{2P} \sum_{h\le H/d} e(\Phi)\,\dd u \end{equation}
in which $\Phi$ abbreviates the expression
$$ \alpha d^3h^3 + 3 h^2\Big(\frac{azd^3}{q} +ud^2\beta\Big) + 3h \Big(\frac{az^2d^3}{q} +u^2d\beta\Big). $$
Since $e(\Phi)$ depends only on $\Phi$ modulo $1$, we may replace $h$ 
by the number $k$ with $1\le k\le q$ and $h\equiv k \bmod q$ in all the fractions
$ad^3h^3/q$, $3ad^3z h^2/q$, $3ad^3z^2h/q$. This shows that
$$ \sum_{h\le H/d}\! e(\Phi) = \sum_{k=1}^q e\Big(\frac{ad^3(k^3+3zk^2+3z^2k)}{q}\Big) \multsum{h\le H/d}{h\equiv k\bmod q}\! e\big(\beta(d^3h^3+3ud^2h^2+3u^2dh)\big). $$ 
In the inner sum, we pick up the congruence condition by additive characters, and then sum over $1\le z\le q$, as in \eqref{314}. This brings in the complete exponential sum
\begin{equation}\label{315} T(q,a,b) = \sum_{k=1}^q \sum_{z=1}^q  e\Big(\frac{a(k^3+3zk^2+3z^2k)+bk}{q}\Big), \end{equation}
and via \eqref{314} we arrive at 
\begin{equation}\label{316} F_w^*(\alpha) = \sum_{d\in D}\frac{\mu(d)}{q^2d} \sum_{-q/2<b\le q/2} T(q,ad^3, b) \int_P^{2P} \sum_{h\le H/d} e(\beta\Psi(u,dh) - bh/q)\,\dd u \end{equation}
where
$$\Psi(u,t) = t^3 + 3ut^2 + 3u^2t. $$

We now concentrate on the inner sum in \eqref{316}, and apply van der Corput's finite version of Poisson's summation formula. In this step, the data $\beta$, $u$, $d$, $b$ and $q$ are fixed but restricted to the ranges implicit in \eqref{316}. The polynomial
$$ \psi(h) = \beta\Psi(u,dh) - bh/q = \beta(d^3h^3 + 3ud^2h^2 + 3u^2dh) -bh/q $$
has derivative
$$ \psi'(h) =\beta(3d^3h^2+6ud^2h+3u^2d) -b/q $$
so that
\begin{align*} |\psi'(h)| &\le \frac{1}{48P^2} (3h^2d^2 +6udh+3u^2) +\frac12 \\ & \le \frac{1}{48P^2}
(3H^2+12PH + 12 P^2) +\frac12 \le \frac78, \end{align*}
at least for large $P$. Further, $\psi''(h) = 6\beta(hd^3+ud^2)$ so that $\psi''(h)$ has the same sign as $\beta$. In particular $\psi'$ is monotonic in the range of summation for $h$. By \cite[Lemma 4.2]{hlm}, we conclude that
\begin{equation}\label{317}
\sum_{h\le H/d} e(\psi(h)) = \sum_{l=-1}^1 \int_0^{H/d} e(\psi(\eta)- l\eta)\,\dd \eta+ O(1), \end{equation} 
uniformly in $\beta$, $u$, $d$, $b$ and $q$ within the ranges implicit in \eqref{316}.

We now substitute \eqref{317} in \eqref{316}. The bounded error in \eqref{317} introduces an error to \eqref{316} that does not exceed
\begin{equation}\label{317a} \ll P \sum_{d\mid w,\,d\le H} \frac{1}{dq^2} \sum_{b=1}^q |T(q,ad^3,b)|. \end{equation}
If we write $x=z+k$ in \eqref{315} and substitute $x$ for $k$, then we recall
\eqref{322} to conclude that
\begin{equation}\label{318} T(q,a,b) = \sum_{x=1}^q \sum_{z=1}^q e\Big(\frac{a}{q}(x^3-z^3) + \frac{b}{q}(x-z)\Big) = |U(q,a,b)|^2.\end{equation}
By orthogonality,
$$ \sum_{b=1}^q |T(q,ad^3,b)| = \sum_{b=1}^q |U(q,ad^3,b)|^2 = q^2. $$
Consequently, the expression in \eqref{317a} is bounded by $O(P\log P)$, and from \eqref{316} and \eqref{318} we infer that, up to an error of size $O(P\log P)$, the sum $F^*_w(\alpha)$ equals
$$
\sum_{d\in D} \frac{\mu(d)}{q^2d} \!\sum_{-\frac12 q<b\le\frac12 q} \!|U(q,ad^3, b)|^2 \sum_{l=-1}^1 \int_P^{2P}\!\!\!\int_0^{H/d}
\! e\Big(\beta\Psi(u,d\eta)- \frac{b\eta}{q} - l\eta\Big)\,\dd \eta\,\dd u. $$
Now recall the conclusion of Lemma \ref{L36} and substitute $c=b+lq$. Then $F_w(\alpha)$, up to an error of size $O(P^{1+\varepsilon})$, equals  
\begin{equation}\label{320} \sum_{d\in D} \frac{\mu(d)}{q^2d}
\!\sum_{-\frac32q<c\le \frac32 q}\!|U(q,ad^3, c)|^2 \int_P^{2P}\!\!\!\int_0^{H/d}
e\Big(\beta\Psi(u,d\eta)- \frac{c\eta}{q}\Big)\,\dd \eta\,\dd u .\end{equation} 

Here, we first consider the term $c=0$ where the integral on the right hand side becomes
$$ \int_P^{2P}\!\!\!\int_0^{H/d}
e(\beta\Psi(u,d\eta))\,\dd \eta\,\dd u = \f1{d} \int_0^H\!\!\!\int_P^{2P}
e(\beta\Psi(u,\eta))\,\dd u\,\dd \eta = \f{K(\beta)}{d}. $$
Thus, the term $c=0$ is exactly the explicit term on the right hand side displayed in the conclusion of Lemma \ref{L37}, and our remaining task is 
to show that terms with $c\neq 0$ contribute an amount no larger than $O(P^{1+\varepsilon})$ to \eqref{320}. 

We wish to apply Lemma \ref{Tit} to the inner integral in  \eqref{320}. The expression
$$ \psi(\eta) = \beta \Psi(u,d\eta) - c\eta/q = \beta(d^3\eta^3 + 3ud^2\eta^2 +3u^2d\eta) - c\eta/q$$
defines a twice differentiable function on $[0,H/d]$, and one has
$$ \psi'(\eta) = d\beta(3d^2\eta^2 + 6ud\eta + 3u^2) - c/q, \quad \psi''(\eta) = d^2\beta (6d\eta +6u). $$
In particular, $\psi$ is monotone. 

Assume now that $d\in D$ is fixed. If the integer $c\neq 0$ has the property that for all $u\in[P,2P]$ and all $\eta\in[0,H/d]$ one has
\begin{equation}\label{Abl}\Big| d\beta(3d^2\eta^2 + 6ud\eta + 3u^2) - \f{c}{q}\Big| \ge \f{|c|}{5q}, \end{equation}
then Lemma \ref{Tit} shows that
$$\Big| \int_0^{H/d} e\Big(\beta \Psi(u, d\eta) - \frac{c\eta}{q}\Big)\, \mathrm d \eta\Big| \le \f{5q}{|c|}, $$ 
uniformly for all $u\in [P,2P]$. Let $C_d^+$ be the set of all integers $c\neq 0$ with $-\f32 q< c\le \f32 q$ where \eqref{Abl} holds for all $u\in [P,2P]$ and all $\eta\in [0,H/d]$. Then, by Lemma \ref{Huasum}, the terms with $c\in C_d^+$ contribute to \eqref{320} an amount not exceeding
\begin{equation}\label{324} \sum_{d\in D} \f{P}{qd} \sum_{c\in C_d^+}\f{|U(q,ad^3,c)|^2}{|c|}
\ll Pq^\varepsilon \sum_{d\in D} \frac1{d} \sum_{1\le |c|\le 2q} \f{(q,c)}{|c|} \ll P^{1+2\varepsilon}. \end{equation} 

Let $C_d^-$ be the set of all non-zero integers $c$ with $-\f32 q<c\le \f32 q$ that are not in $C_d^+$. 
Recalling \eqref{Abl}, we see that for each $c\in C_d^-$, there are real numbers $\eta, u$ with $0\le \eta\le H/d$, $P\le u\le 2P$ and
\begin{equation}\label{Alt} |qd\beta (3d^2\eta^2 + 6ud\eta + 3u^2) - c| < \ts\f15 |c|. \end{equation}

As a first case, we now consider the situation where $qd|\beta| \le 1/(24P^2)$. Then
$$ qd|\beta| (3d^2\eta^2 + 6ud\eta + 3u^2)\le (24P^2)^{-1}(3H^2+12PH+12P^2) \le \ts\f34, $$
at least for large $P$. But then, since $|c|\ge 1$, it follows that the left hand side of \eqref{Alt} is at least as large as $|c|-\f34 > \f15|c|$. This shows that the set $C_d^-$ is empty. Hence, in this case, the proof of Lemma \ref{L37} is already complete because \eqref{324} covers all terms $c\neq 0$ in \eqref{320}.

We may now suppose that $qd|\beta| > 1/(24P^2)$. In addition, we temporarily suppose that $\beta>0$. Then, should \eqref{Alt} hold for some non-zero integer $c$, it follows that $c\ge 1$, and that the number $\lambda$ defined through the equation
$$  qd\beta (3d^2\eta^2 + 6ud\eta + 3u^2) =\lambda c$$
is in the interval $[\f45,\f65]$. We conclude that for large $P$ we have
$$ 1\le c \le \ts \f54 qd\beta \big( 12P^2 +O(PH)\big) \le 16 qd\beta P^2,$$This shows that $C_d^-$ is contained in the interval $[1,16qd\beta P^2]$. Recall that $d\in D$, whence $d\beta \le 1/(48P^2)$, so that we have gained new information here. The total contribution of the sets $C_d^-$ to \eqref{320} now is
\begin{align*} 
 \sum_{d\in D} \frac{\mu(d)}{q^2d}&
\sum_{c\in C_d^-}|U(q,ad^3, c)|^2\int_0^{H/d} e\Big( - \frac{c\eta}{q}\Big)   \int_P^{2P}
e\big(\beta\Psi(u,d\eta)\big)\,\dd u\,\dd \eta \\
\ll & \sum_{d\in D} \frac{1}{q^2d}
\sum_{c\in C_d^-}|U(q,ad^3, c)|^2 \int_0^{H/d} |J(3d^2\eta^2\beta, 3d\eta\beta)|\,\dd\eta.
\end{align*}
Here we apply \eqref{310} and Lemma \ref{Huasum}, and then we substitute $d\eta = v$ to see that the above is bounded by
\begin{align}
\ll & \sum_{d\in D} \frac{1}{q^2d^2} \sum_{1\le c\le 16qd\beta P^2} q^{1+\varepsilon}(q,c) \int_0^H P(1+P^2\beta v)^{-1}\,\dd v \notag\\
\ll & \sum_{d\in D} \frac{q^{2\varepsilon}}{qd^2} (qd\beta P^2) \frac{\log P}{\beta P} \ll P^{1+3\varepsilon}. \label{325}
\end{align}
In this analysis, we supposed that $\beta >0$. If $\beta < 0$, then we may change $\beta$ and $c$ in \eqref{Alt} to $-\beta$ and $-c$ to realize that the same argument applies. By \eqref{324} and \eqref{325} we now see that terms with $c\neq 0$ in \eqref{320} contribute at most $O(P^{1+\varepsilon})$, as required to complete the proof of Lemma \ref{L37}.

\section{Major arcs: the local factors}

In this section we examine the explicit term that remained in the conclusion of Lemma \ref{L37}. It is straightforward to estimate the integral $K(\beta)$ satisfactorily.

\begin{lemma}\label{L38} One has $K(\beta) \ll HP(1+HP^2|\beta|)^{-1} \log P$.
\end{lemma}

\begin{proof} By \eqref{321} and \eqref{37}, and then by \eqref{310}, we find that
$$ |K(\beta)| \le \int_0^H|J(3\beta v^2, 3\beta v)|\,\dd v \ll \int_0^H P(1+P^2|\beta|v)^{-1}  \,\dd v, $$
and the lemma follows.
 \end{proof}
 
The bulk of this section is devoted to the other factor that occurs in Lemma \ref{L37}, and that we now denote by 
\begin{equation}\label{326}
\mathscr S_0(\alpha,w) = 
 \sum_{d\in D(\alpha,w)} \frac{\mu(d)}{q^2d^2}
|U(q,ad^3)|^2. \end{equation}
The first step is to simplify the summation condition to $d\mid w$. Thus, we compare $\mathscr S_0(\alpha,w)$ with
\begin{equation}\label{327}
\mathscr S(\alpha,w) = 
 \sum_{d\mid w} \frac{\mu(d)}{q^2d^2}
|U(q,ad^3)|^2. \end{equation}
\begin{lemma}\label{L39}
Let $\alpha\in\mathfrak M$. Then
$$ \mathscr S_0(\alpha,w)K(\beta)= 
\mathscr S(\alpha,w)K(\beta)
+O(P(\log P)^2). $$
\end{lemma}

\begin{proof} We apply the trivial bound $|U(q,ad^3)|\le q$ and then deduce from \eqref{313}, \eqref{326} and \eqref{327} that
\begin{equation}\label{330}
 | \mathscr S(\alpha,w) - \mathscr S_0(\alpha,w) | \le \sum_d \frac1{d^2}, 
\end{equation}
where the sum on the right hand side extends over all $d\mid w$ where at least one of the inequalities
$$ d>H, \qquad d|\beta|\ge 1/(48P^2) $$
is true. The contribution from all $d$ where $d>H$ is $O(H^{-1})$. We may therefore concentrate on terms where $d\le H$ and $d|\beta|\ge 1/(48 P^2)$.
The contribution of these $d$ to the sum in \eqref{330} does not exceed
\begin{equation}\label{330a} 48P^2 |\beta| \sum_{d\le H} \frac1{d} \ll P^2|\beta| \log P. \end{equation}
We now multiply \eqref{330} with $|K(\beta)|$. For the part of the sum over $d$ in \eqref{330} where $d>H$ it suffices to invoke the trivial bound $|K(\beta)|\le HP$ to see that the corresponding contribution to $|(\mathscr S-\mathscr S_0)K|$ is $O(P)$. The other contribution only arises when $|\beta| \ge 1/(48dP^2)
 \ge 1/(48HP^2)$, and in this case Lemma \ref{L38} gives $K(\beta)\ll (P|\beta|)^{-1}\log P$. When multiplied with the right hand side of \eqref{330a}, we get $P(\log P)^2$. This completes the proof.
\end{proof} 

Recall that when $\alpha\in\mathfrak M$ is given, then $a$ and $q$ are determined by \eqref{38}. By \eqref{327}, we see that $\mathscr S(\alpha,w)$ depends only on $a$ and $q$, and one has
$$ \mathscr S(\alpha,w)=\mathscr S(a/q,w). $$

\begin{lemma}\label{L310} Let $a$ and $q$ be coprime natural numbers. Then $\mathscr S(a/q,w)
\ll \kappa_w(q)^2$.
\end{lemma} 

\begin{proof} We begin by collecting a number of well known results on cubic Gauss sums. First we note that, as an instance of \eqref{H0}, when $q_1$, $q_2$ are coprime natural numbers and $c$ is an integer, then
\begin{equation}\label{331} U(q_1q_2,c) = U(q_1,cq_2^2)U(q_2,cq_1^2). \end{equation}
 Further, on combining the conclusions of Lemmata 4.3, 4.4 and 4.5 of \cite{hlm}, one finds that
\begin{equation}\label{332}  q^{-1}U(q,c) \ll \kappa_1(q). \end{equation}
The references only confirm this estimate when $c$ and $q$ are coprime, but an argument analogous to that delivering \eqref{24} shows that the expression $q^{-1}U(q,c)$ is a function of the fraction $c/q$, and therefore \eqref{332} does not require that $c$ and $q$ be coprime.

Given $q$ and $w$, we put {$w_1$ for the largest divisor of $w$ with $(q,w_1)=1$, and write} $w=w_0 w_1$. Then $(w_0,w_1)=1$, so that each $d$ dividing $w$ factors uniquely as $d=d_0 d_1$ with $d_0\mid w_0$, $d_1\mid w_1$. Furthermore, let $q_1$ denote the largest divisor of $q$ coprime to $w$, and put $q=q_0q_1$. Then the prime factors of $q_0$ divide $w_0$. In particular, $(q_0,q_1)=1$. From \eqref{327} we now infer  
$$ \mathscr S(a/q,w) = \sum_{d_0\mid w_0}\sum_{d_1\mid w_1}
\frac{\mu(d_0)\mu(d_1)}{q_0^2q_1^2 d_0^2 d_1^2} |U(q_0, aq_1^2d^3) U(q_1,aq_0^2d^3)|^2. $$
However, it is readily seen from the definition of $U(q,a)$ that whenever $q$ and $d$ are coprime, then $U(q,ad^3)=U(q,a)$. It follows that $U(q_0, aq_1^2d^3)=U(q_0, aq_1^2d_0^3)$, and similarly for $U(q_1,aq_0^2d^3)$. This shows that $\mathscr S(a/q,w)$ factorises as
\begin{equation}\label{333} \mathscr S(a/q,w) = \mathscr S(aq_1^2/q_0,w_0) \sum_{d_1\mid w_1}\frac{\mu(d_1)}{d_1^2q_1^2} |U(q_1,aq_0^2)|^2. \end{equation}
For the second factor on the right hand side we apply \eqref{332} to deduce that
\begin{equation}\label{334} \sum_{d_1\mid w_1}\frac{\mu(d_1)}{d_1^2q_1^2} |U(q_1,aq_0^2)|^2
= \frac{|U(q_1,aq_0^2)|^2}{q_1^2} \prod_{p\mid w_1} \Big(1-\frac1{p^2}\Big) \ll \kappa_1(q_1)^2. \end{equation}

For the first factor in \eqref{333} we have to work harder. When $r\in\NN$ and $b\in\ZZ$, let
\begin{equation}\label{335} W(r,b) = \multsum{x,y=1}{(x,y,r)=1}^r e\Big(\frac{b}{r}(x^3-y^3)\Big). \end{equation}
Then, by familiar properties of the M\"obius function,
$$ W(r,b) = \sum_{d\mid r} \mu(d)  \multsum{x,y=1}{d\mid x,\,d\mid y}^r e\Big(\frac{b}{r}(x^3-y^3)\Big) = \sum_{d\mid r} \mu(d)|U(r/d,bd^2)|^2. $$
By periodicity, for any $c\in\ZZ$, we have $U(r/d,c)=d^{-1}U(r,cd)$, and then
$$ W(r,b) = \sum_{d\mid r} \f{\mu(d)}{d^2} |U(r,bd^3)|^2. $$
By \eqref{327} we now see that the first factor in \eqref{333} is
\begin{equation}\label{336} \mathscr S(aq_1^2/q_0,w_0) = q_0^{-2} W(q_0,aq_1^2). \end{equation}
Note here that $(a,q)=1$, and hence that $(q_0,aq_1^2)=1$.

The sum $W(r,b)$ is quasi-multiplicative in a manner similar to \eqref{331}.
In fact, when $r=r_1r_2$ with $(r_1,r_2)=1$, the numbers $r_1x_2+r_2x_1$ with $1\le x_j\le r_j$  $(j=1,2)$ meet each residue class modulo $r_1r_2$ exactly once. We now substitute 
$$ x= r_1x_2+r_2x_1, \quad y= r_1y_2+r_2y_1 $$
in \eqref{335}. The condition $(x,y,r)=1$ is equivalent to $(x_1,y_1,r_1)=(x_2,y_2,r_2)=1$, and one observes that
$$ \frac{b(x^3-y^3)}{r_1r_2} \equiv \frac{b}{r_1} r_2^2(x_1^3-y_1^3) +
\frac{b}{r_2} r_1^2(x_2^3-y_2^3) \bmod 1. $$
We then arrive at the identity
\begin{equation}\label{337} W(r_1r_2,b) = W(r_1,r_2^2b) W(r_2,r_1^2 b). \end{equation}

Equipped with \eqref{337}, it now suffices to estimate $W(r,b)$ when $(b,r)=1$ and $r$ is a power of the prime $p$, say. We then express $W(r,b)$ in terms of $U(r,b)$ and the cognate sum
$$ U^*(r,b) = \multsum{x=1}{(x,r)=1}^r e(bx^3/r). $$
In fact, for each $l\ge 1$ one has
$$  \multsum{x,y=1}{(x,y,p)=1}^{p^l} e\Big(\frac{b(x^3-y^3)}{p^l}\Big) = \multsum{x=1}{(x,p)=1}^{p^l} \sum_{y=1}^{p^l} e\Big(\frac{b(x^3-y^3)}{p^l}\Big) + \multsum{x=1}{p\mid x}^{p^l} \multsum{y=1}{(y,p)=1}^{p^l} e\Big(\frac{b(x^3-y^3)}{p^l}\Big), $$
and hence, by \eqref{335},
\begin{equation}\label{338} W(p^l,b) = U^*(p^l,b)U(p^l,-b) +U^*(p^l,-b) \multsum{x=1}{p\mid x}^{p^l}
e\Big(\f{bx^3}{p^l}\Big). \end{equation}
Lemma 4 of Hua \cite{H1} shows that when $p\nmid b$, one has $U^*(p^l,b)=0$ for $p\neq 3$, $l\ge 2$, and for $p=3$, $l\ge 3$. Hence, in these cases, one has $W(p^l,b)=0$, too. 
Next, we take $l=1$ in \eqref{338} and note that $U(p,b) = U^*(p,b)+1$ to confirm that $W(p,b)=|U(p,b)|^2 -1$. By Lemma 4.3 of Vaughan \cite{hlm}, it now follows that when $p\nmid b$, one has
$$ |W(p,b)| \le 4p. $$
We have now estimated or evaluated $W(p^l,b)$ except when $p^l=9$ where we only use the trivial bound $|W(9,b)|\le 72$. We may then summarize our results on $W(p^l,b)$ in the estimate
$$ p^{-2l} |W(p^l,b)| \le \kappa_p(p^l)^2 $$
that is now immediate, recalling the definition of $\kappa_p$. By \eqref{337} and \eqref{336}, it follows that
$$   \mathscr S(aq_1^2/q_0,w_0) \le \kappa_{w_0}(q_0)^2 = \kappa_w(q_0)^2. $$
Finally, by \eqref{333} and \eqref{334}, we infer that
$ \mathscr S(a/q,w) \ll \kappa_w(q)^2$, as required to complete the proof of the lemma.
\end{proof}

We are ready to estimate $F_w(\alpha)$ when $\alpha\in\mathfrak M$. Recall \eqref{326} and \eqref{327}. Now, by Lemmata \ref{L36}, \ref{L37} and \ref{L39}, we have
$$ F_w(\alpha) = \mathscr S(a/q,w)K(\beta) +O(P^{1+\varepsilon}). $$
Then, by Lemmata \ref{L38} and \ref{L310}, we find
$$ F_w(\alpha) \ll \kappa_w(q)^2 HP (\log P) (1+HP^2|\beta|)^{-1} + P^{1+\varepsilon}. $$
This establishes Theorem \ref{thm31} when $\alpha\in \mathfrak M$.

\section{Minor arcs: first reductions}

In this and the next two sections, we establish Theorem \ref{thm31} when \(\alpha \in \mathfrak m\). The initial part of this proof is along familiar lines, and begins with a version of Weyl's inequality. The following lemma is a mild generalization of the lemma in Vaughan \cite{V85}. Recall that \(G(\alpha, X, Y)\) is defined in \eqref{35}.

\begin{lemma}\label{L71} Suppose $1\le Y\le X$, and let $\alpha\in\mathbb R$, $a\in\mathbb Z$, $q\in\mathbb N$ with $|q\alpha-a|\le q^{-1}$ and $(a,q)=1$. Then
$$ G(\alpha)\ll (XYq^{-1/2} + X^{1/2}Y + Y^{1/2}q^{1/2}) (qX)^\varepsilon. $$
\end{lemma} 

\begin{proof}
One has
$$ \Big| \sum_{X<x\le 2X} e(\alpha  h(3x^2+3xh +h^2)\Big|^2 = \sum_{X<x,z\le 2X} e(3\alpha h(x-z)(x+z+h)). $$
We put $x=z+h_1$. As is usual in Weyl differencing, one substitutes $h_1$ for $x$ and carries out the sum over $z$ to find that
$$ \Big| \sum_{X<x\le 2X} e(\alpha  h(3x^2+3xh +h^2)\Big|^2 \ll X + \sum_{1\le h_1< X} \min(X, \|6\alpha hh_1\|^{-1}).$$
We sum this over $h$. Cauchy's inequality and a divisor function estimate  then lead to
\begin{align*}   |G(\alpha)|^2 & \ll Y^2X + Y \sum_{1\le h\le Y} \sum_{1\le h_1<X} \min(X, \|6\alpha hh_1\|^{-1}) \\
& \ll Y^2X + YX^\varepsilon \sum_{1\le l\le 6XY} \min(X, \|\alpha l\|^{-1}).
\end{align*}
The proof is now completed by application of the reciprocal sum lemma \cite[Lemma 2.2]{hlm}.
\end{proof}

Throughout the next two sections, suppose that \(\alpha \in \mathfrak{m}\). By Dirichlet's theorem on diophantine approximation, there exist coprime integers \(a, q\) with \(0 \leq a \leq q \leq 6HP\) and \(|q\alpha - a| \leq (6HP)^{-1}\). We fix one such pair \(a,q\) now. By \eqref{maj}, it follows that 
$$ P < q \leq 6HP,$$
and hence that $a\ge 1$.
Given the parameter \(w\) and \(\alpha \in \mathfrak m\), any {square-free} divisor \(d\) of \(w\) splits as \(d = d_q d_q'\) with \(d_q' = (d, q)\). Then \((d_q, d_q') = 1\).
Now let \(\eta\) denote a positive number, and consider the sets
\begin{align*}
\mathscr{E} &= \{d|w : \mu(d)^2=1, d \leq H, |G(\alpha d^3, P/d, H/d)| \leq P^{1+\eta} d_q^{-1}\},\\
\mathscr{D} &=  \{d|w : \mu(d)^2=1,  d \leq H, |G(\alpha d^3, P/d, H/d)| > P^{1+\eta} d_q^{-1}\}.
\end{align*}
We note that
\begin{equation}\label{72} \multsum{d\mid w}{d\le H} d_q^{-1} \le \sum_{d'_q\mid (w,q)} \multsum{d_q\mid w}{d\le H} d_q^{-1} \ll q^\varepsilon \log P \ll P^{2\varepsilon} \end{equation}
holds uniformly for all  \(w\) composed of prime factors not exceeding \(P\). It follows that
$$ \sum_{d \in \mathscr E} \left| G(\alpha d^3; P/d, H/d) \right| \leq P^{1+\eta} \sum_{d\in \mathscr E} d_q^{-1} \ll P^{1+2\eta}.$$
On combining this estimate with \eqref{36}, we now infer that
\begin{equation}\label{74} |F_w(\alpha)| \leq \sum_{d \mid w} \left| G(\alpha d^3; P/d, H/d) \right| \ll P^{1+2\eta} + \sum_{d \in \mathscr D} \left| G(\alpha d^3; P/d, H/d) \right|.\end{equation}

Next, we investigate the consequences for \(G(\alpha d^3;P/d,H/d)\) when \(d \in \mathscr D\). Set \(R_1 = R_1(d) = P^{3/2+\eta}d_q'd_q^{-1}\). Another application of Dirichlet’s theorem yields coprime numbers \(b = b_d\) and \(r = r_d\) with  $r \leq R_1$ and $|r d^3 \alpha - b| \leq R_1^{-1}\). Then \(\left| r d^3 \alpha - b \right| \leq 1/r\), and Lemma \ref{L71} together with the definition of \(\mathscr D\) delivers the chain of inequalities
\begin{align*}d_q^{-1} P^{1+\eta} 
&< \left| G(\alpha d^3; P/d, H/d) \right| \\&\ll P^{1+\eps} H d^{-2} r^{-1/2} + P^{1+\eps} d^{-3/2} + P^{1/4+\eps} d^{-1/2}r^{1/2} .\end{align*}
However, we may assume that \(\varepsilon < \eta/4\), and then
\begin{align*} P^{1+\eps} d^{-3/2} + P^{1/4+\eps}r^{1/2} d^{-1/2} &<
P^{1+\eta/2} d^{-1} + P^{1/4+\eps}R_1^{1/2} d^{-1/2} \\
&< P^{1+\eta/2} d_q^{-1} + P^{1+3\eta/4} d_q^{-1}. \end{align*}
It transpires that the last two summands on the right hand side of the penultimate display are redundant, and we conclude that
 \(d_q^{-1} P^{1+\eta} \ll P^{1+\eps} H d^{-2} r^{-1/2}\). This last inequality we recast in slightly weaker form as
\begin{equation}\label{75} r_d \ll P^{1-\eta} d^{-4} d_q^{2} \qquad (d \in \mathscr D).\end{equation}

We now apply the Poisson summation formula to $G(\alpha d^3;P/d,H/d)$, for each $d \in \mathscr{D}$. The argument is similar to that used to prove Lemma \ref{L34}. By \eqref{35} and \eqref{37}, we have
\begin{equation}\label{76}G(\alpha d^3 ; X, Y) = \sum_{h \leq Y} e(\alpha d^3 h^3) g(3\alpha d^3 h^2, 3\alpha d^3 h; X),\end{equation}
and for $d \in \mathscr{D}$ one writes
\begin{equation}\label{77}3\alpha d^3 h^2 = \frac{3 b_d h^2}{r_d} + \gamma\end{equation}
in which $\gamma = 3h^2 (\alpha d^3  - b/r)$. But then
\begin{equation}\label{78}|\gamma| \leq 3h^2 (r R_1)^{-1} \leq 3P d^{-2} P^{-3/2-\eta}d_q d_q^{\prime-1}r^{-1} < (2r)^{-1},\end{equation}
at least for large $P$. We now define $\theta$ through the equation
\begin{equation}\label{79}3\alpha d^3 h = \frac{3 b_dh}{r_d} + \theta\end{equation}
and apply Theorem \ref{quadWeyl}. In view of \eqref{77}, \eqref{78}, and \eqref{79}, we infer
\begin{align*}g(3\alpha d^3 h^2, 3\alpha d^3 h; P/d)=  r^{-1} & S(r, 3 bh^2, 3bh) J(\gamma, \theta; P/d)\\& + O((r, 3h)^{1/2} (r + r P d^{-2} |\theta |)^{1/2}).\end{align*}
We simplify the error term. By \eqref{79}, we have
$|\theta| \leq 3h (r R_1)^{-1},$
and hence, by \eqref{75},
\begin{align*}r + r P^2 d^{-2} |\theta| &\ll P^{1-\eta} d^{-4} d_q^2 +(P/d)^{5/2} R_1^{-1}\\&\ll P^{1 - \eta} (d^{-4} d_q^{2}+d^{-5/2}d_q d_q^{\prime-1}) \\&\ll P^{1-\eta} d^{-3/2}.\end{align*}

Now
$$\sum_{h \leq H/d} (h,r)^{1/2}(r + r P^2 d^{-2} |\theta|)^{1/2} \ll P^{(1 - \eta)/2} d^{-3/4} \sum_{h \leq H/d} (h,r)^{1/2} \ll P^{1 - \eta/4} d^{-7/4}.$$
We feed this to \eqref{76} with $X = P/d$, $Y = H/d$, and then sum over $d \in \mathscr{D}$. By \eqref{74}, we then arrive at
$$ |F_w(\alpha)| \ll P^{1 + 2\eta} + \sum_{d \in \mathscr{D}} \sum_{h \leq H/d} \frac{|S(r, 3bh^2,3bh)|}{r}{|J(\gamma, \theta; P/d)|}.$$
We now apply Lemma \ref{LJ} to bound $J$, and Lemma \ref{L24} to control the Gau\ss{} sum. Then we conclude that
\begin{equation}\label{711}|F_w(\alpha)| \ll P^{1 + \eta} + \sum_{d \in \mathscr{D}} \sum_{h \leq H/d} r^{-1/2} (r,h)^{1/2} P d^{-1} \Big(1 + P^2 d^{-2} h \Big|\alpha d^3-\frac{b}{r}\Big|\Big)^{-1}.\end{equation}
The sum over $h$ in \eqref{711} can be carried out with the aid of the following lemma.

\begin{lemma}\label{L72}
Let $K \geq 0$, $H \geq 1$, and $r \in \mathbb{N}$. Then
$$
\sum_{h \leq H} \frac{(r,h)}{(1 + Kh)} \ll \tau(r) (1 + \log H) H (1 + KH)^{-1}.
$$
\end{lemma}

\begin{proof}
If $KH \leq 1$, then the sum in question is bounded above by
$$\sum_{h \leq H} (r,h) \leq \sum_{\substack{d \mid r\\ d\leq H}} d \sum_{\substack{h \leq H\\d\mid h}} 1 \leq \sum_{d \mid r} H.$$
This confirms the lemma in this case. If $KH > 1$, we  see that the sum in question is bounded above by
$$\sum_{h \leq 1/K} (r,h) + \sum_{1/K < h \leq H} \frac{(r,h)}{Kh}.$$
The first summand contributes $\tau(r)/K$, as one sees from the preceding estimate. Applying the same reasoning to the second summand, this term is bounded by
$$\tau(r) \frac{(1 + \log H)}{K}.$$
The proof is complete.
\end{proof}

When applied to \eqref{711}, Lemma \ref{L72} yields
\begin{equation}\label{712}|F_w(\alpha)| \ll P^{1 + 2\eta} + \sum_{d \in \mathscr{D}} \frac{HP^{1+\eps}}{r_d^{1/2}d^{2}} \left(1+HP^2\left|\alpha - b_d/(r_dd^3)\right|\right)^{-1}.\end{equation}
For the $d \in \mathscr{D}$ where
\begin{equation}\label{713}{r_d^{1/2} d^{2}}{(1 + HP^2 |\alpha - b_d/(r_d d^3)|)} \geq H d_q P^{-\eta},\end{equation}
we have recourse to \eqref{72} to see that
$$\sum_{\substack{d \in \mathscr{D}\\\eqref{713} \text{ holds}  }} \frac{HP}{r_d^{1/2} d^{2}} \left(1 +\left|\alpha-b_d/(r_d d^3)\right|\right)^{-1} \leq P^{1 + \eta} \sum_{d \in \mathscr{D}} d_q^{-1} \ll P^{1 +\eta+ \eps}.$$

If we now define $\mathscr{D}' = \mathscr{D}'(\alpha, w)$ to denote the set of all $d \in \mathscr{D}$ where \eqref{713} is \emph{false}, we conclude that \eqref{712} remains valid with $\mathscr{D}'$ in place of $\mathscr{D}$. For $d \in \mathscr{D}'$, we then have
\begin{equation}\label{714}r_d \leq P^{1 - 2\eta} d^{-4} d_q^2,\qquad
\Big|d^3 \alpha - \frac{b_d}{r_d}\Big| \leq r_d^{-1/2} P^{-2 -\eta} dd_q.\end{equation}

For a technical reason that becomes apparent in the next section, we now cover $\mathscr{D}'$ by the two disjoint sets
\begin{equation}\label{714a}\begin{split}\mathscr{D}^\dagger = \{d\in\mathscr{D}: d_q^2d^{-1} \leq (b_d,d^3)P^{1/4}, \eqref{714} \,\text{holds}\},\\
\quad \mathscr{D}^\ddagger = \{d\in\mathscr{D}: d_q^2d^{-1} > (b_d,d^3)P^{1/4},\eqref{714}\,\text{holds}\},
\end{split}\end{equation}
and denote the corresponding portions of \eqref{712} by
\begin{equation}\label{714b}\mathfrak f_w^\dagger = \sum_{d \in \mathscr{D}^\dagger} \frac{r_d^{-1/2} }{d^2} \Big(1 + HP^2 \Big|\alpha - \frac{b_d}{r_d d^3}\Big|\Big)^{-1}\end{equation}
and $\mathfrak f_w^\ddagger$, where the latter is the same expression with $\mathscr{D}^\ddagger$ in the role of $\mathscr{D}^\dagger$. We may then summarize our preparatory analysis as follows.

\begin{lemma}\label{L73}
Let $\alpha \in \mathfrak m$. Then, uniformly in $w$, and in the notation introduced \emph{en passant},  for each \(\eta>0\), one has
$$|F_w(\alpha)| \ll P^{1 + 2\eta} + HP^{1 + \eps} \left(\mathfrak f_w^\dagger + \mathfrak f_w^\ddagger\right).$$
\end{lemma}

\section{Endgame}
In this section, we estimate $\mathfrak f_w^\dagger$ and $\mathfrak f_w^\ddagger$, and we continue to use the notation introduced hitherto. We proceed to show that the approximation $b_d/(r_d d^3)$ to $\alpha$ is independent of $d \in \mathscr{D}^\dagger$. To achieve this, we apply Dirichlet's theorem to find coprime numbers $m$ and $s$ with $1 \leq s \leq P^{5/4}$ and 
$|s\alpha - m| \leq P^{-5/4}.$ 
We now compare the approximations $m/s$ and $b/(r d^3)$ of $\alpha$. By \eqref{714},
$$
\Big|\frac{m}{s} - \frac{b}{r d^3}\Big| \leq \frac{1}{s P^{5/4}} + \frac{d_q}{r^{1/2} d^2 P^{2+\eta}},
$$
and therefore, again by \eqref{714},
\begin{align*}
|mrd^3 - bs| &\leq \frac{r d^3}{P^{5/4}} + \frac{dd_q r^{1/2} s}{P^{2+\eta}} \\
&\leq P^{-1/4 - 2\eta} d^{-1} d_q^2 + P^{-3/2 - \eta}d^{-1} d_q^2 P^{5/4}
\\&\leq 2P^{-1/4 - \eta} d_q^2 d^{-1}.
\end{align*}
Hence, when $d_q^2 d^{-1}\leq (b,d^3)P^{1/4}$ and $P$ is large, we conclude that
$|mrd^3 - bs| < (b,d^3),$ 
so that $b/(rd^3) = m/s$. Note that this fraction is determined by $\alpha$ alone, and that $r_d = s/(s,d^3)$.
It follows that
$$
\mathfrak f_w^\dagger = \sum_{d \in \mathscr{D}^\dagger} \frac{(s,d^3)^{1/2}}{s^{1/2} d^2} \Big(1 + HP^2 \Big|\alpha - \frac{m}{s}\Big|\Big)^{-1}.
$$
Each $d \in \mathscr{D}$ is square-free, and therefore factors uniquely as $d = d' d''$ where $d' \mid s$ and $(d'', s) = 1$. Then
$$
\sum_{d \in \mathscr{D}} \frac{(s, d^3)^{1/2}}{d^2} \leq \sum_{d' \mid s} {d^{\prime-1/2}} \sum_{d''}d^{\prime\prime-2} \ll s^\eps,
$$
and hence,
\begin{equation}\label{81}
\mathfrak f_w^\dagger \ll P^\eps s^{-1/2} \Big(1 + HP^2 \Big|\alpha - \frac{m}{s}\Big|\Big)^{-1}. \end{equation}

\begin{lemma}\label{L81}
Uniformly for $\alpha \in \mathfrak m$ and $w$, one has $\mathfrak f_w^\dagger \ll H^{\eps - 1}$.
\end{lemma}

\begin{proof}
It is only now that the information that $\alpha$ is on minor arcs enters the argument in an essential way.  In fact, when 
$$s^{1/2} \Big(1 + HP^2 \Big|\alpha - \frac{m}{s}\Big|\Big) \ge \frac{1}{6} H,$$
then the desired estimate for $\mathfrak f_w^\dagger$ follows from \eqref{81}. In the opposite case, we have $s \leq \frac{1}{36} H^{2}=\frac{1}{36} P$ and
$$\Big|\alpha - \frac{m}{s}\Big| \leq s^{-1/2} P^{-2} \le (6s HP)^{-1}  (6s^{1/2} H^{-1})
\le  (6s HP)^{-1}.$$
But then $\alpha \in \mathfrak M$, which is not the case. This completes the proof.
\end{proof}

We now turn to $\mathfrak f_w^\ddagger$. As it turns out, when $d_q^2 d^{-1} > (b_d,d^3)P^{1/4}$, the fractions $b_d/(r_d d^3)$ are no longer all the same. Fortunately, the method used to bound $\mathfrak f_w^\dagger$ can be developed further, and thus provides an estimate for the number of distinct values of $b_d/(r_d d^3)$.

We begin to prepare for the relevant counting argument by dyadic slicing of all essential parameters. Given $D \ge 1$, $R \ge 1$, we let $\mathscr{D}^\ddagger(D,R)$ denote the set of all $d \in \mathscr{D}^\ddagger$ where $D \leq d < 2D$, $R \leq r_d < 2R$. Note that $d \in \mathscr{D}$ implies
$$ r_dd^2 \leq P^{1 - 2\eta},$$
as is transparent from \eqref{714}. Hence, $\mathscr{D}^\ddagger(D,R)$ is empty unless $D^2 R \leq P^{1 - 2\eta}$. We denote by $\mathfrak f_w^\ddagger(D,R)$ the portion of the sum defining $\mathfrak f_w^\ddagger$ when $d$ varies over $\mathscr{D}^\ddagger(D,R)$. Then, by a familiar argument, there are numbers $D \geq 1$, $R\geq 1$ with the property that
\begin{equation}\label{82}\mathfrak f_w^\ddagger \ll \mathfrak f_w^\ddagger(D, R) (\log P)^2.\end{equation}

Next, let $\mathscr{D}^\ddagger(D,R,0)$ denote the subset of all $d \in \mathscr{D}^\ddagger(D,R)$ where $|\alpha - b/(r_d d^3)| \leq P^{-5/2}$. Further, when $B \geq 1$, let $\mathscr{D}^\ddagger(D,R,B)$ be the set of all $d \in \mathscr{D}^\ddagger(D,R)$ where
\begin{equation}\label{82a}B P^{-5/2} < |\alpha - b/(r_d d^3)| \leq 2B P^{-5/2}.\end{equation}
By \eqref{714}, the set $\mathscr{D}^\ddagger(D,R,B)$ is empty unless $B \le P^{1/2 - \eta} R^{-1/2}{D}^{-1}$. In view of \eqref{82}, this shows that
\begin{equation}\label{83}\mathfrak f_w^\ddagger \ll \mathfrak f_w^\ddagger(D,R,B) (\log P)^3,\end{equation}
where now $\mathfrak f_w^\ddagger(D,R,B)$ is the portion of the sum defining $\mathfrak f_w^\ddagger$ where $d$ varies over $\mathscr{D}^\ddagger(D,R,B)$ and either $B = 0$ or $B \in [1, P^{1/2 - \eta} R^{-1/2} D^{-1}]$. 
An inspection of \eqref{714b} shows that
\begin{equation}\label{84}\mathfrak f_w^\ddagger(D,R,B) \ll R^{-1/2} D^{-2} (1 + B)^{-1} \# \mathscr{D}^\ddagger(D,R,B).\end{equation}
\begin{lemma}\label{L82}
For all values of the parameters $D, R,$ and $B$ that appear in the preceding argument, one has
$$\# \mathscr{D}^\ddagger(D,R,B) \ll (1 + B) R^2 D^6P^{\eps - 5/2} + D P^{\eps - 1/4}.
$$
\end{lemma}

Equipped with this bound, we deduce from \eqref{84} that
$$\mathfrak f_w^\ddagger(D,R,B) \ll R^{3/2}D^4 P^{\eps - 5/2} + R^{-1/2} D^{-1} P^{\eps - 1/4} (1 + B)^{-1}.$$
By \eqref{714a}, the set $\mathscr{D}^\ddagger(D,R,B)$ is empty unless $D\geq \frac{1}{2} P^{1/4}$. Since $R^{3/2} D^4 \leq P^{2 - \eta}$, it follows that $\mathfrak f_w^\ddagger(D,R,B) \ll P^{\eps - 1/2}$. By \eqref{83}, we then have $\mathfrak f_w^\ddagger \ll P^{\eps - 1/2}$. In view of Lemmas \ref{L73} and \ref{L81}, we now have
$$
F_w(\alpha) \ll P^{1 + 2\eta},
$$
uniformly for $\alpha \in \mathfrak m$ and $w$ a number free of prime factors exceeding $P$. Since the positive number $\eta$ is at our disposal, this argument establishes Theorem \ref{thm31} for $\alpha \in \mathfrak m$.

This leaves the task to prove Lemma \ref{L82}. To lighten notational burdens somewhat, we abbreviate $\# \mathscr{D}^\ddagger(D,R,B)$ to $V$. With $D, R, B$ now fixed, let $V(A)$ denote the number of $d \in \mathscr{D}^\ddagger(D,R,B)$ where
\begin{equation}\label{86}A \leq (b_d,d^3) < 2A.\end{equation}
By \eqref{714a}, we have $V(A) = 0$ for $A > DP^{-1/4}$. By another dyadic slicing, it follows that there is some $A$ in the range $1 \leq A \leq DP^{-1/4}$ where
\begin{equation}\label{87}V \ll V(A) \log P.\end{equation}

We now estimate $V(A)$ in two steps. The first of these is an upper bound for the number, say $\nu$, of the different values among the rational numbers $b_d/(r_d d^3)$ as $d$ varies over the numbers counted by $V(A)$. We temporarily suppose that there are at least two different such values.

We first consider the case where $B = 0$. Then the $\nu \geq 2$ values of $b_d/(r_d d^3)$ all lie in an interval of length $2P^{-5/2}$ centered at $\alpha$. Among these values, we pick two that are closest to each other, and denote these by $b'/(r'd'^3)$ and $b''/(r'' d^{\prime\prime 3})$, where $d', d'' \in \mathscr{D}^\ddagger$, and $r' = r_{d'}, r'' = r_{d''}$, etc. By construction, we have
$$
0 < \left|\frac{b'}{r' d^{\prime 3}} - \frac{b''}{r'' d^{\prime\prime 3}}\right| \leq \frac{2}{(\nu - 1)P^{5/2}}.
$$
Hence
$$
0 < \left| {b'' r' d^{\prime 3} - b' r'' d^{\prime\prime 3}}\right| \leq \frac{2r' r'' (d'd'')^3}{(\nu - 1) P^{5/2}} \leq \frac{2^{9} R^2 D^6}{(\nu - 1) P^{5/2}}.
$$
The integer $b'' r' d^{\prime 3} - b' r'' d^{\prime\prime 3}$ contains the factor $(b',d'^3)(b'',d''^3)$, and this factor is at least as large as $A^2$. It follows that
$$
\nu \ll 1 + R^2 D^6 A^{-2} P^{-5/2}.
$$
This bound also holds when $\nu = 1$.

It is easy to modify the preceding argument when $B \geq 1$. The $\nu \geq 2$ values of $b_d/(r_d d^3)$ then lie in an interval of length $4B P^{-5/2}$ according to the upper bound in \eqref{82a}, and then the argument runs as before, and yields the estimate
\begin{equation}\label{85}\nu \ll 1 + (1 + B) R^2 D^6 A^{-2} P^{-5/2}. \end{equation}
In this form, the bound is valid for all admissible $B$.

In a second step, we estimate how many $d \in \mathscr{D}^\ddagger(D,R,B)$ with \eqref{86} can have the same value of $b_d/(r_d d^3)$. Given one such $d$, let $d'$ be another such number with associated ratio $b'/(r' d'^3)$, and suppose that $b_d/(r_d d^3) = b'/(r' d'^3)$. Then
\begin{equation}\label{88}\frac{b_d}{(b_d, d^3)} r' d^{\prime 3} = b' \frac{r_d d^3}{(b_d, d^3)}.\end{equation}
Recall that $(b_d, r_d) = (b', r') = 1$. Hence $b_d/(b_d,d^3)\mid b'$, and
it follows that $b' = c b_d/(b_d, d^3)$ for some $c \in \mathbb{N}$. The condition \eqref{88} transforms into
\begin{equation}\label{89}r' d^{\prime 3} = c \frac{r_d d^3}{(b_d, d^3)}. \end{equation}
We know that $R D^3 \leq r' d^{\prime 3} \leq 16 R D^3$, and
$$
\frac{R D^3}{2 A} \leq \frac{r_d d^3}{(b_d, d^3)} \leq \frac{16 R D^3}{A}.
$$
This is compatible with \eqref{89} for no more than $O(A)$ values of $c$. Thus, the multiplicity of values $b_d/(r_d d^3)$ is $O(A)$. We have $V(A) \ll \nu A$,
and by \eqref{85}, we conclude that
$$V(A) \ll A P^{\eps} + (1 + B) R^2 D^6 A^{-1} P^{-5/2}.$$
Now $1 \leq A \leq P^{-1/4} D$, and the lemma follows via \eqref{87}.

\section{Le coup de gr\^ace} This section is devoted to the proof of Theorem \ref{exp}. We actually obtain a version of this result that  is concerned with finite sets. Let $N$ be a large natural number, and suppose that $\mathscr Z$ is a set of $Z$ natural numbers contained in the interval $[1,N]$. We take $P=N^{2/5}$, and we let $\rho(n,N)$ denote the number of pairs $(p,z)$ of $z\in\mathscr Z$ and primes $p\in(P,2P]$ satisfying
$ p^3+z = n $. Following the strategy alluded to in the introduction, we proceed to evaluate the moments
\begin{equation}\label{91} M_\nu (N) = \sum_n \rho(n,N)^\nu \end{equation}
when $\nu =1$ or $2$.

Little needs to be said about $M_1$. Using Chebyshev's lower bound, we immediately have
\begin{equation}\label{90} M_1(N) \gg PZ (\log N)^{-1}. \end{equation}
The quadratic moment is the subject of the next lemma.

\begin{lemma}\label{L91} In the notation introduced at the beginning of this section, one has
$$ M_2(N) \ll P^{1+\varepsilon} Z + P^{\varepsilon-1} Z^2. $$
\end{lemma}

\begin{proof}
Note that $M_2(N)$ equals the number of solutions of the equation
$$ p_1^3 - p_2^3 = z_1 -z_2 $$
in primes $p_1,p_2$ in the interval $(P,2P]$, and $z_1,z_2\in\mathscr Z$. The solutions with $p_1=p_2$ contribute less than $PZ$ to $M_2$. In order to count the solutions with $p_1>p_2$, we apply an enveloping sieve. Recall that $\varpi$ denotes the product of all primes not exceeding $P^{1/4}$, and let $M^\dagger(N)$ be the number of solutions of the equation
\begin{equation}\label{92} x_1^3-x_2^3 = z_1-z_2 \end{equation}
in natural numbers $x_1,x_2,z_1,z_2$ constrained by
\begin{equation}\label{93} P<x_1<x_2\le 2P, \quad (x_1,x_2,\varpi)=1, \quad z_1,z_2\in\mathscr Z. \end{equation}
For distinct primes $p_1$, $p_2$ we certainly have $(p_1,p_2,\varpi)=1$. Hence, observing symmetry, we see thus far that
$$ M_2(N) \le PZ + 2M^\dagger(N). $$
Put $x_1=x$ and $x_2-x_1 = h$. Then equation \eqref{92} transforms into
$$ h(3x^2+3xh +h^2) = z_2-z_1. $$ 
By \eqref{93}, we first see that $h\ge 1$, and then that $3x^2+3xh+h^2 > P^2$.
Recalling now that $N=P^{5/2}$, we see that 
the constraints \eqref{93} together with the preceding equation imply conditions
$$ P<x\le 2P, \quad 1\le h\le P^{1/2}, \quad (x,h,\varpi)=1,\quad z_1,z_2\in\mathscr Z. $$
We now apply the circle method to the last equation. Let
$$ Z(\alpha) = \sum_{z\in\mathscr Z} e(\alpha z) $$
and recall \eqref{38}. Then, by orthogonality, the preceding discussion shows
$$  M^\dagger(N) \le \int_0^1 F_\varpi(\alpha)|Z(\alpha)|^2\,\dd\alpha. $$
Again by orthogonality, one observes that
$$ \int_0^1 |Z(\alpha)|^2\,\dd\alpha = Z. $$
We may apply Theorem \ref{thm32} and find that
$$ M^\dagger(N) \ll P^{1+\varepsilon}Z + HP^{1+\varepsilon} \int_{\mathfrak N} \Xi(\alpha) |Z(\alpha)|^2\,\dd\alpha. $$
The pruning lemma (see \cite[Lemma 2]{B89}) yields the bound
$$ \int_{\mathfrak N} \Xi(\alpha) |Z(\alpha)|^2\,\dd\alpha
\ll (HP^2)^{\varepsilon-1} (PZ +Z^2). $$
The lemma now follows, after collecting together our estimates.
\end{proof}
  
\begin{thm}
\label{thm91} In the notation introduced at the beginning of this section, let $\Theta(N)$ denote the number of distinct values that $x^3+z$ takes as the integers $x$ and $z$ vary over $P<x\le 2P$ and $z\in \mathscr Z$. Then
$$ \Theta (N) \gg P^{-\varepsilon} \min\{ PZ, P^3\}. $$
\end{thm} 

\begin{proof} The argument is the same as that in \eqref{15}. We apply Cauchy's inequality in conjunction with \eqref{91},
\eqref{90} and the conclusion of Lemma \ref{L91}. Then
$$ P^{2-\varepsilon}Z^2 \ll M_1(N)^2 \ll M_2(N) \sum_{n: \rho(n,N)\ge 1} 1
\ll  \big( P^{1+\varepsilon} Z + P^{\varepsilon-1} Z^2\big)\Theta(N). $$
This establishes the theorem.
\end{proof}

\begin{proof}[The proof of Theorem \ref{exp}]
Let $\mathscr A$ be a set of exponential density $\delta>0$, and write $\mathscr A(N) = \mathscr A\cap [1,N]$. Suppose that the real number $\delta'$ satisfies $0<\delta'<\delta$. Then, there are infinitely many natural numbers $N$ where $A_N=\# \mathscr A(N) \ge N^{\delta'}$. With such an $N$, we apply Theorem \ref{thm91} with $\mathscr Z = \mathscr A(N)$. Then $\Theta(N)$
counts certain elements of $\mathscr B_3$ (defined in \eqref{11}), and all these elements do not exceed $(2P)^3+N \le 9P^3$, at least when $N$ is large.
It follows that
$$ \#\mathscr B_3(9P^3) \gg  P^{-\varepsilon} \min(P^3, PA_N) \gg P^{-\varepsilon}  \min(P^3, P^{1+5\delta'/2}). $$
It follows that $\delta_3\ge \min(1, \frac13+\frac56\delta')$, for all $\delta'<\delta$. This establishes Theorem \ref{exp}. 
\end{proof}

\vspace{2ex}\noindent
 J\"org Br\"udern \\
 Universit\"at G\"ottingen\\
Mathematisches Institut\\
Bunsenstrasse 3--5\\
D 37073 G\"ottingen \\
Germany\\
jbruede@gwdg.de\\
[2ex]
Simon L. Rydin  Myerson\\
Institutionen f\"or matematiska vetenskaper\\
Avdelningen f\"or Algebra och geometri\\
Chalmers tekniska h\"ogskola och Goteborgs Universitet\\
SE-412 96 G\"oteborg\\
Sweden\\
myerson@chalmers.se

\end{document}